\documentclass[10pt]{article}
   \usepackage{amsmath,amsfonts,amsthm,amssymb,amscd,enumerate, url, anysize, hyperref,slashed}
   \usepackage[pdftex]{graphicx,color}     
\marginsize{1in}{1in}{1in}{1in}

\theoremstyle{plain}
\newtheorem{thm}{Theorem}[section]
\newtheorem{prop}[thm]{Proposition}
\newtheorem{lemma}[thm]{Lemma}
\newtheorem{cor}[thm]{Corollary}

\theoremstyle{definition}

\theoremstyle{remark}
\newtheorem{rmk}[thm]{Remark}

\newcommand{\vol}{\ensuremath{\mathsf{vol}}}

\newcommand{\imag}{\ensuremath{\mathrm{Im}}}
\newcommand{\G}{\ensuremath{\mathrm{G}_2}}
\newcommand{\SP}{\ensuremath{\mathrm{Spin}(7)}}

\newcommand{\R}{\ensuremath{\mathbb R}}
\newcommand{\C}{\ensuremath{\mathbb C}}
\newcommand{\Qu}{\ensuremath{\mathbb H}}

\newcommand{\Oc}{\ensuremath{\mathbb O}}
\newcommand{\ph}{\ensuremath{\varphi}}

\newcommand{\st}{\ensuremath{\ast}}

\newcommand{\tr}{\ensuremath{\operatorname{Tr}}}

\newcommand{\spn}{\ensuremath{\operatorname{span}}}

\newcommand{\rest}[2]{\ensuremath{{#1} |_{{}_{#2}}}}
\newcommand{\be}{\ensuremath{\bar e}}
\newcommand{\ce}{\ensuremath{\check e}}
\newcommand{\bn}{\ensuremath{\bar \nu}}
\newcommand{\cn}{\ensuremath{\check \nu}}

\newcommand{\co}{\ensuremath{\check \omega}}
\newcommand{\nup}{\ensuremath{\nu^{\perp}}}
\newcommand{\oo}{\ensuremath{\mathbf 1}}
\newcommand{\oi}{\ensuremath{\mathbf i}}
\newcommand{\oj}{\ensuremath{\mathbf j}}
\newcommand{\ok}{\ensuremath{\mathbf k}}
\newcommand{\oee}{\ensuremath{\mathbf e}}
\newcommand{\oie}{\ensuremath{\mathbf i \mathbf e}}
\newcommand{\oje}{\ensuremath{\mathbf j \mathbf e}}
\newcommand{\oke}{\ensuremath{\mathbf k \mathbf e}}
\newcommand{\spp}{\ensuremath{{\slashed{\mathcal S}_{\!+}}}}
\newcommand{\spm}{\ensuremath{{\slashed{\mathcal S}_{\!-}}}}
\newcommand{\spi}{\ensuremath{{\slashed{\mathcal S}}}}

\newcommand{\jm}{\ensuremath{\mathbf j_{\scriptscriptstyle L}}}
\newcommand{\inner}[1]{\left\langle #1 \right\rangle}

\newcounter{commentCounter}
\numberwithin{equation}{section}
\numberwithin{table}{section}
\numberwithin{figure}{section}

\begin{document}

\title{Deformations of calibrated subbundles \\ of Euclidean spaces via twisting by special sections}

\author{Spiro Karigiannis\footnote{The research of the first author is partially supported by a Discovery Grant from the Natural Sciences and Engineering Research Council of Canada. Portions of this research were conducted while the second author was an undergraduate student research assistant of the first author in summer 2009, and was supported by the Chinese University of Hong Kong.} \\ {\it Department of Pure Mathematics, University of Waterloo} \\ \tt{karigiannis@math.uwaterloo.ca} \and Nat Chun-Ho Leung \\ {\it Department of Pure Mathematics, University of Waterloo} \\ \tt{nat.leung@uwaterloo.ca} }

\maketitle

\begin{abstract}
We extend the ``bundle constructions'' of calibrated submanifolds, due to Harvey--Lawson in the special Lagrangian case, and to Ionel--Karigiannis--Min-Oo in the cases of exceptional calibrations, by ``twisting'' the bundles by a special (harmonic, holomorphic, or parallel) section of a complementary bundle. The existence of such deformations shows that the moduli space of calibrated deformations of these ``calibrated subbundles'' includes deformations which destroy the linear structure of the fibre.
\end{abstract}

\tableofcontents

\section{Introduction} \label{introsec}

In this paper we examine some explicit deformations through calibrated submanifolds of calibrated ``subbundles'' of Euclidean spaces. Let $M$ be a Riemannian manifold with a calibration, that happens to also be the total space of a vector bundle over a base $Q$. Then a \emph{calibrated subbundle} $N$ of $M$ is a calibrated submanifold of $M$ which is also a subbundle of $M$, in the sense that $N$ is the total space of a vector bundle over a submanifold $P$ of $Q$, whose fibres are subspaces of the corresponding fibres of $M$. The following are some examples.

If $L^p$ is a $p$-dimensional austere submanifold of $\R^n$, then the total space of its conormal bundle $N^* L$ is an $n$-dimensional special Lagrangian submanifold in $T^* \R^n \cong \C^n$. Similarly, if $L^2$ is a $2$-dimensional submanifold of $\R^4$ which is minimal (or negative superminimal) then the bundle $\Lambda^2_- (\R^4)$ of anti-self dual $2$-forms on $\R^4$ restricts on $L^2$ to the direct sum $E \oplus F$ of a rank $1$ and a rank $2$ real vector bundle over $L^2$, respectively, and the total space of $E$ is associative (or the total space of $F$ is coassociative) in $\Lambda^2_- (\R^4) \cong \R^7$. A similar construction holds for Cayley subbundles of $\R^8$ as rank $2$ real vector bundles over a minimal surface $L^2$ in $\R^4$, obtained by restricting the negative spinor bundle $\spm (\R^4) \cong \R^8$ of negative chirality spinors on $\R^4$ to the submanifold $L^2$ and decomposing the restriction into the direct sum of two rank $2$ real vector bundles over $L^2$. The total space of each one of these is a Cayley submanifold of $\R^8$. The  construction in the special Lagrangian case is due to Harvey--Lawson~\cite{HL}, while the constructions in the case of the exceptional calibrations appeared in Ionel--Karigiannis--Min-Oo~\cite{IKM}. All these constructions were later generalized to noncompact manifolds of special holonomy that are total spaces of vector bundles over compact bases, such as the Stenzel manifolds $T^* S^n$, which admit Calabi-Yau metrics, and the Bryant--Salamon manifolds of $\G$ or $\SP$ holonomy, by Karigiannis--Min-Oo~\cite{KM}.

The purpose of the present paper is the following. In 1993, a generalization of the Harvey--Lawson conormal bundle construction of special Lagrangian submanifolds in $\C^n$ was presented by Borisenko~\cite{Bo}. This construction involves ``twisting'' the conormal bundle by the gradient of a smooth function $\rho$ on the austere submanifold $L^p \subset \R^n$, and finding the condition on $\rho$ for the resulting smooth $n$-dimensional submanifold of $\C^n$, which is \emph{no longer} a vector subbundle of $\rest{T^* \R^n}{L}$, to be special Lagrangian. In the case when $p=2$ the function $\rho$ needs to be a harmonic function on $L$. For $p > 2$, the condition is more complicated, and was only considered by Borisenko for $p=3$ and $n=4$. First, we rederive the Borisenko construction using the notation of~\cite{IKM}, but for general $p$ and $n$. We also extend his construction by considering a twisting by a closed $1$-form $\mu$, rather than an exact $1$-form $d\rho$. We then proceed to adapt this idea to give an analogous construction of twisted calibrated subbundles in the setting of the exceptional calibrations on $\R^7$ and $\R^8$. The main results in this paper are contained in Theorems~\ref{SLthm},~\ref{assthm},~\ref{coassthm}, and~\ref{cayleythm}.

These new examples of calibrated submanifolds of Euclidean spaces are deformations of calibrated subbundles which are no longer total spaces of vector bundles. This shows that the moduli space of calibrated submanifolds near a calibrated subbundle includes both deformations of the base $L$ (as a submanifold of the required ``type'' for the associated vector bundle to be calibrated), \emph{and} deformations of the ``fibre'' in a way that destroys the linear structure, but remains a foliation of smooth submanifolds foliated by the original base $L$ of the calibrated subbundle. \emph{In particular this answers, in the negative, the question posed at the very end of~\cite{KM} about whether calibrated subbundles can only deform as bundles.}

The deformation theory of \emph{compact} calibrated submanifolds was first studied by McLean~\cite{Mc}. His arguments used the Hodge theory of compact oriented Riemannian manifolds extensively, in particular the $L^2$-orthogonal decomposition of the space of smooth forms into harmonic forms plus exact forms plus coexact forms. To study the moduli space of \emph{noncompact} calibrated submanifolds, one needs noncompact analogues of the Hodge theorem, which are much more complicated. However, in the case of noncompact oriented Riemannian manifolds which are \emph{asymptotically cylindrical} or \emph{asymptotically conical}, for example, then the techniques of Lockhart--McOwen~\cite{LM, L} can be employed. Much work has been done on the deformation theory of noncompact calibrated submanifolds which are asymptotically conical or asymptotically cylindrical. A partial list of references to such work includes~\cite{J, JS, Lo1, Lo2, Pac}.

The outline of our paper is as follows. Sections~\ref{borisenkosec} and~\ref{exceptionalsec} discuss our ``twisted'' calibrated subbundle constructions in the special Lagrangian and exceptional cases, respectively. Section~\ref{examplesec} presents an explicit example, and in Section~\ref{conclusionsec} we summarize some of the more important observations that can be made and questions that can be asked. Finally the Appendix collects a useful lemma on the symmetric polynomials of a matrix and the octonion multiplication table, which are used in the text to prove the main theorems.

\begin{rmk}
As the present paper is in some sense a sequel to both~\cite{IKM} and~\cite{KM}, we use the notation established in those papers throughout. Readers may find it helpful to familiarize themselves with~\cite{IKM} before reading the present paper. Although the proofs of the main theorems are somewhat similar to those in~\cite{IKM}, we provide as much detail as possible for completeness. The special Lagrangian case (Theorem~\ref{SLthm}) is the most different, and the equations~\eqref{SLthmeq} derived there may prove to be interesting in and of themselves.
\end{rmk}

{\bf Acknowledgements.} The authors would like to thank Ho-Yeung Hung for helping to check some of the original calculations, in the special Lagrangian and the coassociative cases, in the summer of 2009. The first author would also like to thank Jason Lotay for useful discussions.

\section{Review and extension of the Borisenko construction} \label{borisenkosec}

In this section we review and extend the Borisenko generalization~\cite{Bo} of the Harvey--Lawson conormal bundle construction of special Lagrangian submanifolds of $\C^n$, for a general $p$-dimensional submanifold $L^p$ of $\R^n$.

Let $\{e_1, \ldots, e_p\}$ be a local orthonormal frame of tangent vectors to $L$, and let $\{e^1, \ldots, e^p\}$ be the dual coframe for the cotangent bundle. Similarly, let $\{\nu_1, \ldots, \nu_{n-p}\}$ be a local orthonormal frame of normal vector fields to $L$ and let $\{ \nu^1, \ldots, \nu^{n-p} \}$ be the dual coframe for the conormal bundle. By parallel transport using the tangential and normal connections, we can assume without loss of generality that for a fixed point $x \in L$, we have
\begin{equation} \label{assumptioneq}
(\nabla_{e_i} e_j)|^{T} _x \, \, \text{\, and \, } (\nabla_{e_i} \nu_j)|^{N}_x =0
\end{equation}
where $\nabla$ is the Levi-Civita connection on $\R^n$. If $\nu$ is any normal vector field on $L$, we define the second fundamental form in the direction of $\nu$ by
\begin{align*}
A^{\nu} : \, \, & T_x M \rightarrow T_x M \\ A^{\nu}(w) = & \, (\nabla_w \nu)^T
\end{align*}
for any tangent vector $w$ to $L$ at $x$.
\begin{rmk} \label{signconventionrmk}
As mentioned in~\cite{IKM}, here we follow the sign convention of Harvey--Lawson~\cite{HL}, which differs from the more widely used convention.
\end{rmk}
For notational convenience we will denote
\begin{equation*}
A^{k}_{ij } = A^{\nu_k}_{ij} = \inner{A^{\nu_k}(e_i), e_j} = A^{k}_{ji}.
\end{equation*}
Now at the point $x$, by our assumption, $\nabla_{e_i} e_j$ has no tangential component. Therefore \emph{at the point $x$}, we have
\begin{equation*}
\nabla_{e_i} e_j = \sum_{k=1}^{n-p} \inner{\nabla_{e_i}e_j, \nu_k} \nu_k = -\sum_{k=1}^q A^{k}_{ij} \nu_k,
\end{equation*}
where we denote $q = n - p$. Similarly, we also have, \emph{at the point $x$}, that
\begin{equation*}
\nabla_{e_i}\nu_j = \sum_{k=1}^{p} \inner{\nabla_{e_i}\nu_j, e_k} e_k = \sum_{k=1}^{p} A^{j}_{ik} e_k.
\end{equation*}
From these two formulas it follows immediately that at $x$, we have
\begin{equation} \label{secondfundformeqs}
\nabla_{e_i}e^j = - \sum_{k=1}^{q} A^k_{ij} \nu^k, \qquad \qquad \nabla_{e_i} \nu^j = \sum_{k=1}^p A_{ik}^j e^k.
\end{equation}
Let $N^*L$ be the conormal bundle of $L$ in $\R^n$. It was shown in~\cite{HL} that $N^*L$ is special Lagrangian in $T^* \R^n$ with a particular phase if and only if $L$ is austere. That is, the odd degree symmetric polynomials of $A^\nu$ vanish for every normal vector field $\nu$. More generally, we will consider the following situation. Let $\mu$ be a smooth $1$-form on $L$ and define
\begin{equation} \label{Xmudefneq}
X_{\mu} \, = \, \{(x, \xi + \mu_x) \in \rest{T^* \R^n}{L} : \, x \in L, \, \xi \in N^*_x L\}.
\end{equation}
This is a ``twisting'' of the conormal bundle $N^*L$ obtained by affinely translating each fibre $N^*_x L$ by a vector $\mu_x$, in the orthogonal complement $T^*_x L$, which varies with $x \in L$. Of course, for $\mu = 0$ we recover the conormal bundle. It is clear that $X_{\mu}$ is a smooth $n$-dimensional submanifold of $T^* R^n \cong \C^n$. We can also mnemonically write $X_{\mu}$ as ``$N^*L + \mu$''.

In order to state our theorem, we need to recall the definition of the elementary symmetric polynomials $\sigma_k(A)$ of a $p \times p$ matrix $A$. These can be defined by
\begin{equation} \label{sigmakeq}
\det(I + t A) \, = \, \sum_{k=0}^p t^k \sigma_k(A).
\end{equation}
The cases $\sigma_1(A) = \tr(A)$ and $\sigma_p(A) = \det(A)$ are the most familiar, but all these matrix invariants play an important role in special Lagrangian geometry. Note that $\sigma_0(A) = 1$.

\begin{prop} \label{SLprop}
The submanifold $X_\mu$ is Lagrangian in $T^* \R^n$ if and only if $d\mu = 0$.
\end{prop}
\begin{proof}
Although this is a simple calculation, and a well known result, it is easy to get confused by the notation, since we are looking at the total space of a vector bundle, so we proceed carefully. We need to show that every tangent space to $X_{\mu} = N^*L + \mu$ is a Lagrangian subspace of the corresponding tangent space to $T^* \R^n$ if and only if $\mu$ is closed.
In local coordinates $(u^1, \ldots, u^p)$ for $L$ and coordinates $(t_1, \ldots, t_q)$ for the fibres of $N^* L$ with respect to the local trivialization $\{\nu^1, \ldots, \nu^q\}$, the immersion $h$ of $X_{\mu}$ in $T^* \R^n$ is given by
\begin{equation*}
h: (u^1, u^2, \ldots, u^p, t_1, t_2, \ldots, t_q) \mapsto \left( x^1(\mathbf u), \ldots, x^n(\mathbf u), t_1 \nu^1 + t_2 \nu^2 + \ldots + t_q \nu^q + \mu(x^1(\mathbf u), \ldots, x^n(\mathbf u) ) \right) .
\end{equation*}
A basis for the tangent space to $X_{\mu}$ at the point $h(\mathbf u_0, t_1, t_2, \ldots, t_q)$ is given by the vectors
\begin{equation} \label{EFvectseq}
\begin{aligned}
E_i & = \, h_* \left(\frac{\partial}{\partial u^i}\right) = \left( e_i , \, \sum_{k=1}^q t_k \rest{(\nabla_{e_i} \nu^k)}{x(\mathbf u_0)}  + \rest{(\nabla_{e_i} \mu)}{x(\mathbf u_0)} \right) & & i = 1, \ldots, p, \\ F_j & = \, h_* \left(\frac{\partial}{\partial t_j}\right) = (0, \nu^j) = \cn^j & & j = 1, \ldots, q.
\end{aligned}
\end{equation}
Since $T_{(x,\alpha_x)}(T^* \R^n) \cong T_x \R^n \oplus T_x \R^n$ naturally, as in~\cite{IKM} we will denote $(e_i, 0)$ by $\be_i$ and $(0, e^i)$ by $\ce^i$. Then $\be^k$ is dual to $\be_k$ and $\ce^k$ is dual to $\ce_k$. With this notation and equation~\eqref{secondfundformeqs} we can write
\begin{align} \nonumber
E_i & = \, \left( e_i , \sum_{k=1}^q \sum_{l=1}^p t_k A^k_{il} e^l + \sum_{l=1}^p (\nabla_{e_i} \mu)(e_l) e^l +  \sum_{l=1}^q (\nabla_{e_i} \mu)(\nu_l) \nu^l \right) \\ \nonumber & = \, \be_i + \sum_{l=1}^p A^{\nu}_{il} \ce^l + \sum_{l=1}^p (\nabla_{e_i} \mu)(e_l) \ce^l + \sum_{l=1}^q (\nabla_{e_i} \mu)(\nu_l) \cn^l \\ & \label{Evecteq2} = \, \be_i + \sum_{l=1}^p \left( A^{\nu}_{il} + (\nabla_{e_i} \mu)(e_l) \right) \ce^l + \sum_{l=1}^q (\nabla_{e_i} \mu)(\nu_l) \cn^l
\end{align}
where we have defined $\nu = \sum_{k=1}^q t_k \nu_k$.

In this basis, the canonical symplectic form $\omega$ on $T^* \mathbb R^n$ is given by
\begin{equation} \label{omegaeq}
\omega = \sum_{k=1}^p \be^k \wedge \ce_k + \sum_{l=1}^q \bn^l \wedge \cn_l.
\end{equation}
Hence, to check when this immersion is Lagrangian, we use \eqref{omegaeq} and compute
\begin{equation*}
\omega(F_i, F_j) = \omega( \cn^i, \cn^j ) = 0 \qquad \forall \, i,j = 1, \ldots, q,
\end{equation*}
and (dropping the summation sign over $k$ for clarity) we also have
\begin{equation*}
\omega(F_i, E_j) = \omega( \cn^i, \be_j + \left( A^{\nu}_{jk} + (\nabla_{e_j} \mu)(e_k) \right) \ce^k + (\nabla_{e_j} \mu)(\nu_k) \cn^k ) = 0 \qquad \forall \, i = 1, \ldots, q, \quad \forall \, j = 1, \ldots, p.
\end{equation*}
Finally (again with the summations over $k$ and $l$ implied) we find that
\begin{align*}
\omega(E_i, E_j) & = \, \omega( \be_i + \left( A^{\nu}_{il} + (\nabla_{e_i} \mu)(e_l) \right) \ce^l + (\nabla_{e_i} \mu)(\nu_l) \cn^l , \be_j + \left( A^{\nu}_{jk} + (\nabla_{e_j} \mu)(e_k) \right) \ce^k + (\nabla_{e_j} \mu)(\nu_k) \cn^k ) \\ & = \, A^{\nu}_{ij} + (\nabla_{e_i} \mu)(e_j) - A^{\nu}_{ji} - (\nabla_{e_j} \mu)(e_i) = 2 (d\mu)(e_i, e_j)
\end{align*}
using the symmetry of $A^{\nu}$ and the fact that the exterior derivative is the skew-symmetrization of the covariant derivative. Thus we see that $X_{\mu}$ is Lagrangian if and only if $\mu$ is closed.
\end{proof}

We now ask when $X_{\mu}$ is special Lagrangian. The result of Theorem~\ref{SLthm} below is quite complicated, but we make several observations about the theorem immediately following its proof.
\begin{thm} \label{SLthm}
Suppose $d\mu = 0$, so that $X_{\mu}$ is Lagrangian. Let $B$ be the \emph{symmetric} matrix
\begin{equation*}
B_{ij} = (\nabla_{e_i} \mu)(e_j) = \frac{1}{2} \left( (\nabla_{e_i} \mu)(e_j) + (\nabla_{e_j} \mu)(e_i) \right).
\end{equation*}
Thus $B$ is the matrix of the \emph{symmetrized covariant derivative} of $\mu$.  Let $\phi = \frac{\pi}{2} q - \theta$. Then $X_\mu$ is special Lagrangian with phase $e^{i \theta}$ if and only if
\begin{equation} \label{SLthmeq}
\begin{aligned}
& & \imag \left(e^{i \phi} \det(I + i B) \right) & = \, 0, \\ \text{and} & & & \\
& & \sigma_j (A^{\nu} (I + i B)^{-1} ) & = \, (-1)^j \, \sigma_j (A^{\nu} (I - i B)^{-1}) \quad \forall j = 1, \ldots, p,
\end{aligned}
\end{equation}
for every normal vector field $\nu$, with corresponding second fundamental form $A = A^{\nu}$.\end{thm}
\begin{proof}
Since a basis for the $(1,0)$ forms is given by $\be^j + i \ce_j$ for $j = 1, \ldots, p$ and $\bn^k + i \cn_k$ for $k = 1, \ldots, q$, the holomorphic $(n,0)$ form $\Omega$ on $T^* \R^n$ is
\begin{equation} \label{Omegaeq}
\Omega = (\be^1 + i \ce_1) \wedge \ldots \wedge(\be^p + i \ce_p) \wedge (\bn^1 + i \cn_1) \wedge \ldots \wedge (\bn^q + i \cn_q).
\end{equation}
From~\eqref{EFvectseq} and~\eqref{Evecteq2}, we have 
\begin{align*}
(\be^j + i \ce_j)(E_i) & = \, \delta_{ji} + i \lambda_i A_{ji} + i (\nabla_{e_i} \mu)(e_j), & 
(\bar{e}^j + i\check{e}_j) (F_i) & = \, 0, \\
(\bar{\nu}^j + i \check{\nu}_j)(E_i) & = \, i (\nabla_{e_i} \mu)(\nu_j) , & 
(\bar{\nu}^j + i \check{\nu}_j) (F_i) & = \. i \delta_{ji}.
\end{align*}
Thus by~\eqref{Omegaeq}, we have
\begin{equation} \label{Omegatempeq}
\Omega(E_1,...,E_p,F_1,...,F_q) = i^q \det( \delta_{ji} + i A_{ji} + i  \nabla_{e_i} \mu(e_j)) = i^q \det ( I+i (A + B)).
\end{equation}
As in  \cite{HL}, changing the point $(t_1,...,t_q)$ to $(st_1,...,st_q)$ results in changing $A$ to $sA$. Now for $X_{\mu}$ to be special Lagrangian in $\C^n$ with phase $e^{i \theta}$, we need $\rest{\imag(e^{-i\theta} \Omega)}{X_{\mu}} = 0$, and hence at each point $x$ in $L$, and for each normal direction $\nu$, we must have that
\begin{equation} \label{fseq}
f(s) \, = \, \imag \left(e^{-i \theta} i^q \det (I+i(sA + B) ) \right) = 0 \quad \forall \, s.
\end{equation}

\begin{rmk} \label{phasermk}
Before concluding the proof, we should comment on the possible phase $e^{i\theta}$. In the case considered by Harvey--Lawson~\cite{HL} and reviewed in~\cite{IKM}, we had $\mu = 0$, and thus $B = 0$. Since the real part of $\det(I + isA)$ is always nonzero for any $s$, in this situation we must take $e^{i\theta} = \pm i^q$. (The minus sign just corresponds to a change of orientation.) However, in the general case $B \neq 0$,  the constant term (corresponding to $s=0$) in $\det(I + i(sA + B))$ is $\det(I + iB) = \sum_{k=0}^p i^k \sigma_k(B)$, and it is no longer true in general (if $p \geq 2$) that the real part of this is always nonzero. In fact, for any choice of phase $e^{i\theta}$, we can get a differential equation for $\mu$ alone by setting $f(0) = 0$ which must be satisfied.
\end{rmk}

Returning to the proof, let $e^{-i\theta} i^q = e^{i \phi}$, where $\phi = \frac{\pi}{2} q - \theta$. By equation~\eqref{fseq} we need
\begin{equation*}
e^{i \phi} \det( I + i (B + s A) ) - e^{-i \phi} \det( I - i (B + s A) ) \, = \, 0 \quad \forall \, s,
\end{equation*}
which by~\eqref{sigmakeq} becomes
\begin{equation*}
e^{i\phi} \sum_{k=0}^p (i)^k \sigma_k (B+sA) - e^{-i\phi} \sum_{k=0}^p (-i)^k \sigma_k (B+sA) \, = 0 \, \quad \forall \, s.
\end{equation*}
Because this is a $p^{\text{th}}$ order polynomial in $s$, it vanishes identically if and only if the first $p$ derivatives in $s$, at $s=0$, all vanish. Since $B$ is symmetric, the eigenvalues are real, and thus $I \pm i B$ is always invertible. Hence we can apply Lemma~\ref{appendixlemma} to the above expression with $t = \pm i$, and obtain
\begin{equation*}
e^{i \phi} j! \, i^j \det(I + iB) \, \sigma_j (A (I + i B)^{-1} ) \, = \, e^{-i \phi} j! \, (-i)^j \det(I - i B) \, \sigma_j (A (I - i B)^{-1}) \quad \forall j = 0, \ldots, p,
\end{equation*}
which simplifies to
\begin{equation} \label{SLtempeq2}
e^{2 i \phi} \det(I + iB) \, \sigma_j (A (I + i B)^{-1} ) \, = \, (-1)^j \det(I - i B) \, \sigma_j (A (I - i B)^{-1}) \quad \forall j = 0, \ldots, p.
\end{equation}
Now let $j=0$ in~\eqref{SLtempeq2}. Since $\sigma_0(C) = 1$ for any $C$, we get
\begin{equation} \label{SLtempeq3}
e^{2i \phi} \det(I + i B) = \det(I - i B),
\end{equation}
which, substituted back into~\eqref{SLtempeq2}, gives
\begin{equation*}
\sigma_j (A (I + i B)^{-1} ) \, = \, (-1)^j \, \sigma_j (A (I - i B)^{-1}) \quad \forall j = 1, \ldots, p,
\end{equation*}
the second part of~\eqref{SLthmeq}. Finally, equation~\eqref{SLtempeq3}, corresponding to $j=0$, can also be rewritten as
\begin{equation*}
2i \, \imag \left(e^{i \phi} \det(I + i B) \right) \, = \, e^{i \phi} \det(I + i B) - e^{-i \phi} \det(I - i B) \, = \, 0,
\end{equation*}
and is the first part of~\eqref{SLthmeq}.
\end{proof}

Let us make some observations about Theorem~\ref{SLthm}.

(\emph{The case $\mu=0$.}) If $\mu = 0$, then $B = 0$, and~\eqref{SLthmeq} reduces to $e^{i\phi} = \pm 1$ and
\begin{equation*}
\sigma_j (A) = (-1)^j \, \sigma_j (A) \quad \forall j = 1, \ldots, p.
\end{equation*}
That is, we recover the result of Harvey--Lawson~\cite{HL} that the conormal bundle $N^* L$ is special Lagrangian in $T^* \R^n$ with phase $i^q$ if and only if all the odd degree symmetric polynomials in the eigenvalues of $A^{\nu}$ vanish for all normal vector fields $\nu$ on $L$. Such a submanifold $L$ of $\R^n$ is called \emph{austere}.

\bigskip

Now consider the second part of~\eqref{SLthmeq}. When $j = p$, using $\sigma_p = \det$ and the first part of~\eqref{SLthmeq}, we get
\begin{equation} \label{jpeq}
e^{2 i \phi} \det(A^{\nu}) \, = \, (-1)^p \det(A^{\nu})
\end{equation} 
for every normal vector field $\nu$. In fact this can also be seen from~\eqref{SLtempeq2} using the multiplicativity of the determinant. Since the right hand side of~\eqref{jpeq} is real, we see that, unless \emph{every} $A^{\nu}$ is singular, we must have $e^{i \phi} \in \{ \pm 1, \pm i \}$, and depending on the parity of $p$ and whether $e^{2i\phi}$ is $+1$ or $-1$, this either gives no information or tells us that indeed, each $A^{\nu}$ is singular. Meanwhile the first part of~\eqref{SLthmeq} can be rewritten as
\begin{equation} \label{jzeroeq}
\sin \phi \left( 1 - \sigma_2 (B) + \sigma_4 (B) - \sigma_6 (B) + \cdots \right) \, = \, \cos \phi \left( \sigma_1 (B) - \sigma_3 (B) + \sigma_5 (B) - \cdots \right).
\end{equation}
This equation is formally \emph{identical} to the equation satisfied by a special Lagrangian graph in $\C^n$, derived by Harvey--Lawson~\cite{HL}. Note that
\begin{equation*}
\sigma_1(B) \, = \, \sum_{k=1}^p (\nabla_{e_k} \mu)(e_k) \, = \, - d^* \mu,
\end{equation*}
so $\sigma_1(B) = 0$ is precisely the condition that the closed $1$-form $\mu$ be \emph{coclosed}, and hence harmonic.

We can also simplify the second part of~\eqref{SLthmeq} when $j=1$ as follows. Without loss of generality, we can assume that our oriented orthonormal local frame $\{ e_1, \ldots, e_p \}$ for $L$ has been chosen so that the symmetric matrix $B$ is diagonal: $B_{kl} = \lambda_k \delta_{kl}$, with real eigenvalues $\lambda_k$. Then~\eqref{SLthmeq} for $j=1$ becomes:
\begin{align*}
\tr (A(I + i B)^{-1}) & = \, -\tr (A(I - i B)^{-1}) \\ \sum_{k=1}^p \sum_{l=1}^p A_{kl} (1 + i \lambda_l \delta_{lk})^{-1} & = \, - \sum_{k=1}^p \sum_{l=1}^p A_{kl} (1 - i \lambda_l \delta_{lk})^{-1},
\end{align*}
which can be easily rearranged to obtain
\begin{equation} \label{joneeq}
\sum_{k=1}^p \frac{A_{kk}}{ 1 + \lambda_k^2} \, = \, 0.
\end{equation}
Similar expressions can be obtained for the second part of~\eqref{SLthmeq} when $j=2, \ldots, p-1$ using the fact that $\sigma_j(C) = \tr(\Lambda^j C)$, but these expressions are not particularly enlightening.

\bigskip

(\emph{The case $p=1$.}) When $p=1$ ($L$ is a curve), and in any codimension $q$, equations~\eqref{SLthmeq} reduce to only~\eqref{jpeq} and~\eqref{jzeroeq} with $p=1$. These become
\begin{equation*}
e^{2i \phi} A^{\nu} \, = \, - A^{\nu}, \qquad \qquad \sin \phi \, = \, - \cos \phi \, d^* \mu.
\end{equation*}
The case $e^{i\phi} = \pm i$ gives the contradiction $1 = 0$, so we must have $A^{\nu} = - A^{\nu}$, and $d^* \mu = -\tan \phi$. So $L^1$ is a minimal $1$-dimensional submanifold of $\R^n$ (hence a straight line) and $\mu$ is a closed $1$-form on $L \cong \R$ satisfying $\Delta_L \mu = d d^* \mu + d^* d \mu = -d (\tan \phi) + d^*(0) = 0$, since $\phi$ is constant. In fact, since $H^1(\R) = 0$, we know that $\mu = df$ for some function $f$ on $L$ satisfying $\Delta_L f = \tan \phi$. If we choose coordinates on $\R^n$ so that $L^1$ is just the $x = x^1$ axis, then $\mu = (a x + b) dx$ for some constants $a$ and $b$, and thus $X_{\mu} = $ ``$N^* L + \mu$'' is just an affine translation of $N^* L$ in $\C^n = \R^n \oplus \R^n$, and is thus an $n$-plane.

\bigskip

(\emph{The case $p=2$.}) When $p=2$ ($L$ is a surface), and in any codimension $q$, equations~\eqref{SLthmeq} reduce to~\eqref{jpeq},~\eqref{jzeroeq}, and~\eqref{joneeq} with $p=2$. These become
\begin{equation*}
e^{2i \phi} \det A^{\nu} \, = \, \det A^{\nu}, \qquad \sin \phi \, (1 - \sigma_2 (B) ) \, = \, \cos \phi \, \tr (B), \qquad \frac{A^{\nu}_{11}}{1 + \lambda_1^2} + \frac{A^{\nu}_{22}}{1 + \lambda_2^2} \, = \, 0.
\end{equation*}
If some $\det A^{\nu} \neq 0$, then $e^{i \phi} = \pm 1$, and thus $d^* \mu = -\tr(B) = 0$. Hence $\lambda_1 = - \lambda_2$, and then the third equation above gives $\tr A^{\nu} = 0$, for any $\nu$, and hence $L$ is a minimal surface in $\R^n$, and $\mu$ is a harmonic $1$-form on $L$. These are the only conditions. On the other hand, if every $\det A^{\nu} = 0$, then the phase $e^{i \phi}$ can be arbitrary, and the second equation above becomes $1 - \sigma_2 (B) = \cot \phi \, \tr (B)$, which is much more complicated. Given a solution of this equation, we can then substitute back into the third equation above to find conditions on the second fundamental form of $L$ in $\R^n$.

We can summarize part of the above discussion as follows.
\begin{cor} \label{SLptwocor}
When $p=2$, then $X_{\mu}$ is special Lagrangian with phase $i^q$ if and only if $L$ is minimal in $\R^n$ and $\mu$ is a \emph{harmonic} $1$-form on $L$.
\end{cor}

\begin{rmk} \label{compareborisenkormk}
The results in this section are extensions of the work of Borisenko~\cite{Bo}. He considered only the special case when $\mu = d\rho$ is \emph{exact}, and $n=3$, $p=2$, with fixed phase $i^q = i$.
\end{rmk}

\begin{rmk} \label{KMquestionrmk}
From the above discussion, it appears likely that there exist solutions to~\eqref{SLthmeq} in which $L$ is not austere in $\R^n$. This would give a \emph{negative} answer to ``Question 5.3'' in~\cite{KM}, for the special Lagrangian case. See also the brief discussion in Section~\ref{conclusionsec} for more about this.
\end{rmk}

\section{Analogous constructions for the exceptional calibrations} \label{exceptionalsec}

In this section we present similar constructions of calibrated submanifolds of $\R^7$ and $\R^8$ which are deformations of total spaces of vector bundles, obtained by ``twisting'' the constructions of~\cite{IKM}.

\subsection{Associative and coassociative submanifolds of $\R^7$} \label{g2sec}

We begin by reviewing (see~\cite{IKM}), that $\Lambda ^2_- (\R^4) \cong \R^7$ has a canonical torsion-free $\G$-structure $\ph$. Let $\{ e^1, e^2, e^3, e^4 \}$ be any local oriented coframe for $\R^4$. Then each fibre of $\Lambda ^2_- (\R^4)$ is spanned by 
\begin{align*}
\omega^1 & = \, e^1 \wedge e^2 - e^3 \wedge e^4, \\
\omega^2 & = \, e^1 \wedge e^3 - e^4 \wedge e^2, \\
\omega^3 & = \, e^1 \wedge e^4 - e^2 \wedge e^3.
\end{align*}
An oriented orthonormal frame for the total space $\Lambda^2_-(\R^4)$ is given by
\begin{equation*}
(e_i , 0) \qquad \text{ for } i = 1, \ldots, 4 \qquad \text{ and } \qquad (0 , \omega^i) \qquad \text{ for } i = 1, \ldots, 3.
\end{equation*}
To simplify the notation, we will denote $(e_i, 0)$ by $\be_i$ and $(0, \omega^i)$ by $\co^i$. The canonical $\G$-structure $\ph$ on $\Lambda^2_- (\mathbb R^4)$ is then given by
\begin{equation} \label{g2formeq}
\begin{aligned}
\ph \, = \, \, & \co_1 \wedge \co_2 \wedge \co_3 + \co_1 \wedge ( \be^1 \wedge \be^2 - \be^3 \wedge \be^4 ) \\ & {}+ \co_2 \wedge ( \be^1 \wedge \be^3 - \be^4 \wedge \be^2 ) + \co_3 \wedge ( \be^1 \wedge \be^4 - \be^2 \wedge \be^3 )
\end{aligned}
\end{equation}
where $\co_k$ is dual to $\co^k$ and $\be^k$ is dual to $\be_k$.

Now we restrict the bundle $\Lambda^2_-(\R^4)$ to an oriented surface $L^2$ in $\R^4$, with $\{e^1, e^2\}$ an oriented local coframe for $L$, and $\{e^3, e^4 \} = \{\nu^1, \nu^2\}$ an oriented local frame for the conormal bundle. We also assume that at a fixed point $x \in L$, the frames have been chosen to satisfy the equations~\eqref{assumptioneq}.
\begin{prop} \label{selfdualformulasprop}
In such an adapted local frame, at the point $x$, we have
\begin{align*}
\nabla_{e_i} \omega^1 & = \, (A^2 _{i1} - A^1_{i2}) \omega^2 + (-A^1_{i1} - A^2_{i2}) \omega^3, \\
\nabla_{e_i} \omega^2 & = \, (A^1 _{i2} - A^2 _{i1}) \omega^1, \\
\nabla_{e_i} \omega^3 & = \, (A^2_{i2} + A^1_{i1}) \omega^1.
\end{align*}
\end{prop}
\begin{proof}
See~\cite[Proposition 4.1.2]{IKM}.
\end{proof}
Notice that, when restricted to $L^2$, we have $\omega^1 = \vol_L - \st_{\R^4} \vol_L$ is a globally defined nowhere vanishing section of $\rest{\Lambda^2_-(\R^4)}{L}$. Hence $\rest{\Lambda^2_-(\R^4)}{L}$ can be decomposed as $E \oplus F$, where $E$ is the real line bundle over $L$ spanned by $\omega^1$, and $F = E^{\perp}$ is a rank $2$ real vector bundle over $L$ locally spanned by $\omega^2$ and $\omega^3$. In~\cite{IKM}, it is proved that the total space of $F$ is \emph{associative} in $\R^7$, and the total space of $F$ is \emph{coassociative} in $\R^7$, if and only if $L$ is minimal or negative superminimal in $\R^4$, respectively.

Following the strategy of Section~\ref{borisenkosec}, it is natural to consider the following. Let $\sigma$ be a section of the bundle $F$ over $L$. Define $X_{\sigma}$ by
\begin{equation} \label{Xsigmadefneq}
X_{\sigma} \, = \, \{(x, \eta + \sigma_x) \in \rest{\Lambda^2_-(\R^4)}{L} : \, x \in L, \, \eta \in E_x \}.
\end{equation}
As in Section~\ref{borisenkosec}, this is a ``twisting'' of the bundle $E$ over $L$ obtained by affinely translating each fibre $E_x$ by a vector $\sigma_x$, in the orthogonal complement $F_x$, which varies with $x \in L$. We will mnemonically write $X_{\sigma}$ as ``$E + \sigma$''. We want to find conditions on the immersion of $L$ in $\R^4$ and on the section $\sigma$ of $F$ so that $X_{\sigma}$ is an associative submanifold of $\R^7$.

Before stating our theorem, we make the following observations. First, we note that $L^2$ is an oriented Riemannian $2$-manifold, and hence is a complex $1$-dimensional K\"ahler manifold, with complex structure $J$ defined locally by $J e^1 = e^2$ and $J e^2 = - e^1$. Also, since $F$ is a rank $2$ real vector bundle over $L$, with an orientation given by $\{ \omega^2, \omega^3 \}$, it is actually a rank $1$ \emph{complex} vector bundle over $L$, with complex structure given locally by $J \omega^2 = \omega^3$ and $J \omega^3 = - \omega^2$. Since $L$ is complex one dimensional, this $F$ is actually a \emph{holomorphic} line bundle (as there are no $(0,2)$-forms.)

\begin{thm} \label{assthm}
The submanifold $X_{\sigma}$ is an associative submanifold of $\R^7 \cong \Lambda^2_- (\R^4)$ if and only if $L$ is minimal in $\R^4$ and $\sigma$ is a \emph{holomorphic section} of $F$.
\end{thm}
\begin{proof}
We need to check when every tangent space to $X_{\sigma}$ is an associative subspace of the corresponding tangent space to $\Lambda^2_- (\R^4)$. In local coordinates the immersion $h$ is
\begin{align*}
h : (u^1, u^2, t_1) \quad \mapsto \quad & (x^1(u^1, u^2), x^2(u^1, u^2), \, t_1 \omega^1 + \sigma(u^1, u^2) ) \\ = \quad & (x^1(u^1, u^2), x^2(u^1, u^2), \, t_1 \omega^1 + \alpha(u^1, u^2) \omega^2 + \beta(u^1, u^2) \omega^3).
\end{align*}
Here $\alpha$ and $\beta$ are locally defined smooth functions which are the coordinates of $\sigma$ with respect to the local trivialization $\{ \omega^2, \omega^3 \}$ of $F$. We omit the explicit dependence of each $\omega^i$ on $(u^1, u^2)$ for notational simplicity. 
Thus the tangent space to $X_{\sigma}$ at $(x (\mathbf u_0), t_1 \omega^1 + \sigma)$ is
spanned by the vectors
\begin{align*}
E_i & = \, h_* \left(\frac{\partial}{\partial u^i}\right) = \left( e_i, \, t_1 \nabla_{e_i} (\omega^1)  + \nabla_{e_i} (\alpha \omega^2 + \beta \omega^3) \right), \qquad i = 1, 2, \\ F_1 & = \, h_* \left(\frac{\partial}{\partial t_1}\right) = (0, \omega^1) = \co^1.
\end{align*}
Using Proposition~\ref{selfdualformulasprop}, we find that
\begin{equation*}
E_i = \, \be_i + a_i \co^1 +b_i \co^2 + c_i \co^3,
\end{equation*}
where
\begin{equation*}
a_i \, = \, \alpha (A^1 _{i2} - A^2 _{i1}) + \beta (A^2_{i2} + A^1_{i1}), \qquad b_i \, = \, t_1 (A^2 _{i1} - A^1_{i2}) + \alpha_i, \qquad c_i \, = \, t_1 (-A^1_{i1} - A^2_{i2}) + \beta_i,
\end{equation*}
with $\alpha_i = \frac{\partial \alpha}{\partial u^i}$ and $\beta_i = \frac{\partial \beta}{\partial u^i}$. To check when the tangent space at $(x (\mathbf u_0), t_1 \omega^1 + \sigma)$ is associative, we need to see when the octonion associator $[E_1, E_2, F_1] = (E_1 E_2)F_1 - E_1(E_2 F_1)$ vanishes, where at this point, without loss of generality, we can make the following explicit identification of the tangent space of $\Lambda^2_- (\R^4)$ with $\imag (\Oc)$:
\begin{equation*}
\begin{pmatrix} \co^1 & \co^2 & \co^3 & \be_1 & \be_2 & \bn_1 & \bn_2 \\ \updownarrow & \updownarrow & \updownarrow & \updownarrow & \updownarrow & \updownarrow & \updownarrow \\ \oi & \oj & \ok & \oee & \oie & \oje & \oke \end{pmatrix}.
\end{equation*}
Therefore we have
\begin{align*}
E_1 & = \,  \oee + a_1 \oi + b_1 \oj + c_1 \ok \\ E_2 & = \,  \oie + a_2 \oi + b_2 \oj + c_2 \ok \\ F_1 & = \, \oi.
\end{align*}
Now a tedious computation (see the octonion multiplication table in Appendix~\ref{octtablesec}) gives
\begin{align*}
[E_1,E_2,F_1] & = \, (E_1 E_2) F_1 - E_1 (E_2 F_1) \\ & = 2(c_2 - b_1) \oje - 2 (b_2 + c_1) \oke \\ & = \, -2 \left( (A^2_{11} + A^2_{22}) t_1 + (\beta_2 - \alpha_1) \right) \oje + 2 \left( ( A^1_{11} +  A^1_{22}) t_1 - (\alpha_2 + \beta_1) \right) \oke.
\end{align*}
This will vanish for all $x \in L$ and all $t_1 \in \R$ if and only if $\tr A^{\nu_1} = \tr A^{\nu_2} = 0$ (that is, $L$ is minimal) \emph{and} $\alpha_1 = \beta_2$ and $\alpha_2 = - \beta_1$. All that remains is to verify that these two equations on $\alpha$ and $\beta$ are equivalent to the holomorphicity of the section $\sigma$. Let $\nabla^F$ denote the connection on $F$ induced from $\nabla$ on $\R^4$. Since $F$ is a holomorphic line bundle, we see that
\begin{equation*}
\bar \partial_F \sigma \, = \, 0 \, \Leftrightarrow \, (\nabla^F)^{(0,1)} \sigma \, = \, 0 \, \Leftrightarrow \, (e_1 + i e_2) \left( \nabla^F \sigma \right) \, = \, 0.
\end{equation*}
Let $\pi_F$ denote the orthogonal projection from $E \oplus F$ onto $F$. We have
\begin{equation} \label{holomorphiceq}
\begin{aligned}
(e_1 + i e_2) \left( \nabla^F \sigma \right)  & = \, \nabla^F_{e_1} \sigma + J( \nabla^F_{e_2} \sigma) \\ & = \, \pi_F (\nabla_{e_1} \sigma) + J( \pi_F (\nabla_{e_2} \sigma) ) \\ & = \, \pi_F (\alpha_1 \omega^2 + \alpha \nabla_{e_1} \omega^2 + \beta_1 \omega^3 + \beta \nabla_{e_1} \omega^3) + J ( \pi_F (\alpha_2 \omega^2 + \alpha \nabla_{e_2} \omega^2 + \beta_2 \omega^3 + \beta \nabla_{e_2} \omega^3)) \\ & = \, \alpha_1 \omega^2 + \beta_1 \omega^3 + J(\alpha_2 \omega^2 + \beta_2 \omega^3) \\ & = \, (\alpha_1 - \beta_2) \omega^2 + (\beta_1 + \alpha_2) \omega^3,
\end{aligned}
\end{equation}
where we have used Proposition~\ref{selfdualformulasprop} once again. Thus the pair of equations $\alpha_1 = \beta_2$ and $\alpha_2 = - \beta_1$ are equivalent to $\bar \partial_F \sigma = 0$. This completes the proof.
\end{proof}

Similarly we can look for coassociative submanifolds by twisting the vector bundle $F$ by a section of $E$. Specifically, let $\eta$ be a section of the trivial real line bundle $E$ over $L$. Define $X_{\eta}$ by
\begin{equation} \label{Xtaudefneq}
X_{\eta} \, = \, \{(x, \eta_x + \sigma) \in \rest{\Lambda^2_-(\R^4)}{L} : \, x \in L, \, \sigma \in F_x \}.
\end{equation}
Again, this is a ``twisting'' of the bundle $F$ over $L$ obtained by affinely translating each fibre $F_x$ by a vector $\eta_x$, in the orthogonal complement $E_x$, which varies with $x \in L$. We will mnemonically write $X_{\eta}$ as ``$\eta + F$''. We want to find conditions on the immersion of $L$ in $\R^4$ and on the section $\eta$ of $E$ so that $X_{\eta}$ is a coassociative submanifold of $\R^7$.

\begin{thm} \label{coassthm}
The submanifold $X_{\eta}$ is a coassociative submanifold of $\R^7 \cong \Lambda^2_- (\R^4)$ if and only if $L$ is negative superminimal in $\R^4$ and $\tau$ is a \emph{parallel section} of $E$, with respect to the connection $\nabla^E$ on $E$ induced from $\nabla$ on $\R^4$.
\end{thm}
\begin{proof}
We need to determine when every tangent space to $\eta + F$ is a coassociative subspace of the corresponding tangent space to $\Lambda^2_-(\R^4)$. In local coordinates the immersion $h$ is given by
\begin{equation*}
h : (u^1, u^2, t_2, t_3) \mapsto (x^1(u^1, u^2), x^2(u^1, u^2), \, \gamma(u^1, u^2) \omega^1+ t_2 \omega^2 + t_3 \omega^3),
\end{equation*}
where $\eta = \gamma \omega^1$ for some smooth \emph{globally defined} function $\gamma$ on $L$, since $E$ is trivialized by the global section $\omega^1$. As before, we omit the explicit dependence of each $\omega^i$ on $(u^1, u^2)$. Thus the tangent space to $X_{\eta}$ at the point $(x (\mathbf u_0), \eta + t_2 \omega^2 + t_3 \omega^3)$ is
spanned by the vectors
\begin{align*}
E_i & = \, h_* \left(\frac{\partial}{\partial u^i}\right) = \left( e_i , \nabla_{e_i}(\gamma \omega^1) +  t_2 \nabla_{e_i} (\omega^2) + t_3 \nabla_{e_i} (\omega^3) \right) \qquad i = 1, 2, \\ F_j & = \, h_* \left(\frac{\partial}{\partial t_j}\right) = ( 0, \omega^j) = \co^j \qquad j = 2, 3.
\end{align*}
Using Proposition~\ref{selfdualformulasprop} we find that
\begin{equation*}
E_i = \, \be_i + a_i \co^1 +b_i \co^2 + c_i \co^3,
\end{equation*}
where
\begin{equation*}
a_i \, = \, \gamma_i + t_2 ( A^1_{i2} - A^2_{i1} ) + t_3 ( A^2_{i2} + A^1_{i1} ), \qquad b_i \, = \, \gamma (A^2 _{i1} - A^1_{i2}) , \qquad c_i \, = \, \gamma (-A^1_{i1} - A^2_{i2}),
\end{equation*}
with $\gamma_i = \frac{\partial \gamma}{\partial u^i}$. As in~\cite{IKM}, we define the vectors $\nu(t_2, t_3) = t_2 \nu_1 + t_3 \nu_2$ and $\nup(t_2, t_3) =
- t_3 \nu_1 + t_2 \nu_2$, which are orthogonal normal vectors. Then the expressions for $a_i$ simplifies to
\begin{equation*}
a_i \, = \, \gamma_i + (A^{\nu}_{i2} - A^{\nup}_{i1}). \\
\end{equation*}
Now since we have
\begin{align*}
\ph \, = \, \, & \co_1 \wedge \co_2 \wedge \co_3 + \co_1 \wedge (\be^1 \wedge \be^2 - \bn^1 \wedge \bn^2 ) \\ & {}+ \co_2 \wedge (\be^1 \wedge \bn^1 - \bn^2 \wedge \be^2 ) + \co_3 \wedge (\be^1 \wedge \bn^2 - \be^2 \wedge \bn^1 )
\end{align*}
we can check when the immersion is coassociative by determining when $\ph$ restricts to zero on each of these tangent spaces. A computation gives\begin{align*}
\ph (E_1, E_2, F_2) & = \, a_2 c_1 - a_1 c_2, & \ph (E_1, E_2, F_3) & = \, a_1 b_2 - a_2 b_1, \\ 
\ph (F_2, F_3, E_1) & = \, a_1, & \ph (F_2, F_3, E_2) & = \, a_2.
\end{align*}
Hence these all vanish if and only if $a_1 = a_2 = 0$. Replacing $(t_2, t_3)$ by $(\lambda t_2, \lambda t_3)$ changes $A^{\nu}$ to $\lambda A^{\nu}$ and $A^{\nup}$ to $\lambda A^{\nup}$, and thus $a_1 = a_2 = 0$ for all $x \in L$ and all $t_2, t_3 \in \R$ if and only if
\begin{equation*}
A^{\nu}_{12} - A^{\nup}_{11} = 0, \qquad A^{\nu}_{22} - A^{\nup}_{12} = 0, \qquad \gamma_1 = 0, \qquad \gamma_2 = 0.
\end{equation*}
As explained in~\cite{IKM}, the first two equations above say that $L$ is negative superminimal in $\R^4$, while the last two equations say that $\gamma$ is a constant function on $L$. Hence we find that
\begin{equation*}
\nabla^E_{e_i} (\eta) = \pi_E (\nabla_{e_i} \eta) = \pi_E (\gamma_i \omega^1 + \gamma \nabla_{e_i} \omega^1) = 0
\end{equation*}
using Proposition~\ref{selfdualformulasprop}. Thus $\eta$ is a parallel section of $E$ with respect to $\nabla^E$.
\end{proof}
\begin{rmk} \label{coassrmk}
In~\cite{IKM}, it is shown that when $L$ is negative superminimal in $\R^4$, the section $\omega^1$ of $\rest{\Lambda^2_- (\R^4)}{L}$ is a parallel section, and the coassociative submanifolds constructed there are actually complex surfaces lying inside a $\C^6$ in $\R^7$. Theorem~\ref{coassthm} says that these can only be twisted by a constant multiple of $\omega^1$, and are thus just affine translates of the examples from~\cite{IKM}.
\end{rmk}

\subsection{Cayley submanifolds of $\R^8$} \label{spin7sec}

In this section we consider the $\SP$-manifold $\R^8 \cong \spm (\R^4)$, the bundle of negative chirality spinors over $\R^4$, and its Cayley submanifolds. We begin by briefly reviewing some of the facts discussed in~\cite{IKM}. Writing the octonions as $\Oc \cong \Qu \oplus \Qu \oee$, the fibre of spinors ${\spi}_x$ over $x \in \R^4$ is isomorphic to $\Oc$, with $(\spp)_x \cong \Qu \oee$ and $(\spm)_x \cong \Qu$. In addition the cotangent space $T^*_x \R^4$ iis also identified with $\Qu \oee$. With these identifications, the Clifford product of a cotangent vector in $T^*_x \R^4$ with a spinor in $\spi_x$ is given by octonionic multiplication. Explicitly, the representation is given by
\begin{align*}
\gamma & \, : \, \, T^* \R^4 \to \operatorname{End}(\spp \oplus \spm) \\
\gamma(\alpha) (s) & \, = \, \alpha s 
\end{align*}
where $\alpha$ is a $1$-form, $s \in \spp \oplus \spm$ and the product $\alpha s$ is octonionic multiplication. Since $\Oc$ is \emph{not associative}, we need to be careful when composing two elements of this representation:
\begin{equation*}
\left(\gamma(\alpha_1) \gamma(\alpha_2)\right) (s) = \gamma(\alpha_1) \left( \gamma(\alpha_2)(s) \right) =  \gamma(\alpha_1) (\alpha_2 s) = \alpha_1 ( \alpha_2 s)
\end{equation*}
which in general is \emph{not} the same as $(\alpha_1 \alpha_2) s$.

Now, if $L^2$ is an oriented submanifold of $\R^4$, then the restriction $\rest{\spm (\R^4)}{L}$ splits naturally into the direct sum of two rank $2$ real vector bundles $V_+ \oplus V_-$ over $L$. This can be seen as follows. Define $r = \gamma(e^1) \gamma(e^2)$ for any oriented orthonormal basis of $L$. It is easy to see that $r$ is well defined, and in~\cite{IKM} it is shown that $r$ is a linear endomorphism of $\spm$ such that $r^2 = -1$, so $r$ is a complex structure on $\rest{\spm (\R^4)}{L}$, and $V_+$ and $V_-$ are defined to be the $\pm i$ eigenspaces of $r$. In fact the map $r$ is given, using $(\spm)_x \cong \Qu$, by right multiplication by the unit imaginary quaternion $\jm = e^1 e^2$ where $e^1, e^2 \in T^*_x \R^4 \cong \Qu \oee$. Thus at each point $x \in L$, we have $(V_+)_x = \spn { \{\oo, \jm \} }$ and $(V_-)_x = (\spn { \{\oo, \jm \} })^{\perp}$ in $\Qu$.

 In~\cite{IKM}, it is proved that the total space of $V_+$ is \emph{Cayley} in $\R^8$ if and only if $L$ is minimal. (This is true for $V_-$ as well.) As in the two previous sections, we want to consider the natural twisted version of this construction. Let $\psi$ be a section of the bundle $V_-$ over $L$. Define $X_{\psi}$ by
\begin{equation} \label{Xpsidefneq}
X_{\psi} \, = \, \{(x, \chi + \psi_x) \in \rest{\spm(\R^4)}{L} : \, x \in L, \, \eta \in (V_+)_x \}.
\end{equation}
This is a ``twisting'' of the bundle $V_+$ over $L$ obtained by affinely translating each fibre $(V_+)_x$ by a vector $\psi_x$, in the orthogonal complement $(V_-)_x$, which varies with $x \in L$. We will mnemonically write $X_{\psi}$ as ``$V_+ + \psi$''. We want to find conditions on the immersion of $L$ in $\R^4$ and on the section $\psi$ of $V_-$ so that $X_{\psi}$ is a Cayley submanifold of $\R^8$.

Before stating our theorem, we make the following observations. As in Section~\ref{g2sec}, the submanifold $L$ is a K\"ahler manifold of complex dimension one. Also, $V_+ \oplus V_-$ is a quaternionic line bundle on $L$, which is identified with a $\C^2$-bundle over $L$ by the complex structure $r$. However, we can also think of each $V_{\pm}$ as a complex line bundle over $L$, with respect to a different complex structure. Specifically, the identification $(V_+)_x = \spn { \{\oo, \jm \} }$ makes $V_+$ into an $\mathrm{SO}(2) \cong \mathrm{U}(1)$ bundle, with complex structure $J_+$ on $V_+$ given by $J_+ (\oo) = \jm$ and $J_+ (\jm) = - \oo$. Then $V_-$, being the orthogonal complement of $V_+$ in the $\mathrm{SO}(4)$ bundle $\rest{\spm (\R^4)}{L}$, is also a complex line bundle. Since $L$ is complex one dimensional, both $V_+$ and $V_-$ are actually \emph{holomorphic} line bundles over $L$.

\begin{thm} \label{cayleythm}
The submanifold $X_{\psi}$ is a Cayley submanifold of $\R^8 \cong \spm (\R^4)$ if and only if $L$ is minimal in $\R^4$ and $\psi$ is a \emph{holomorphic section} of $V_-$.
\end{thm}
\begin{proof}
We need to determine when every tangent space to $X_{\psi}$ is a Cayley  subspace of the corresponding tangent space to $\spm (\R^4)$. In local coordinates the immersion $h$ is
\begin{equation*}
h : (u^1, u^2, t_1, t_2) \mapsto (x^1(u^1, u^2), x^2(u^1, u^2), \, t_1 q_1(u^1,u^2) + t_2 q_2(u^1,u^2) + \alpha(u^1, u^2) q_3 + \beta(u^1, u^2) q_4)
\end{equation*}
where $q_1$ and $q_2$ are a local oriented orthonormal frame for $V_+$ and $q_3$ and $q_4$ are a local oriented orthonormal frame for $V_-$. Here $\psi = \alpha q_3 + \beta q_4$. We omit the explicit dependence of each $q_i$ on $(u^1, u^2)$ for notational simplicity.

The tangent space to $X_{\psi}$ at $(x (\mathbf u_0), \, t_1 q_1 + t_2 q_2 + \psi)$ is spanned by the vectors
\begin{equation} \label{cayleytangentvectoreq1}
\begin{aligned}
E_k & = \, h_* \left(\frac{\partial}{\partial u^k}\right) = e_k + \nabla_{e_k} (t_1 q_1 + t_2 q_2) + \nabla_{e_k}(\alpha q_3 + \beta q_4), \qquad k = 1, 2, \\ F_k & = \, h_* \left(\frac{\partial}{\partial t_k}\right) = q_k, \qquad k = 1, 2.
\end{aligned}
\end{equation}
In~\cite{IKM}, an expression is derived for $\nabla_{e_k} q_j$ for $j=1,2$. The exact same argument, with an extra minus sign, gives $\nabla_{e_k} q_j$ for $j=3,4$. The results are:
\begin{equation} \label{cayleytangentvectoreq2}
\begin{aligned}
\nabla_{e_k} q_j & = \, \frac{\jm}{2} \left( A^1_{k1} \gamma(\nu^1) \gamma(e^2) + A^2_{k1} \gamma(\nu^2) \gamma(e^2) + A^1_{k2} \gamma(e^1) \gamma(\nu^1) + A^2_{k2} \gamma(e^1) \gamma(\nu^2) \right) q_j, \quad j = 1,2, \\ \nabla_{e_k} q_j & = \, -\frac{\jm}{2} \left( A^1_{k1} \gamma(\nu^1) \gamma(e^2) + A^2_{k1} \gamma(\nu^2) \gamma(e^2) + A^1_{k2} \gamma(e^1) \gamma(\nu^1) + A^2_{k2} \gamma(e^1) \gamma(\nu^2) \right) q_j, \quad j = 3,4.
\end{aligned}
\end{equation}
where we have used the notation $A^k_{ij} = \langle e_i, A^{\nu_k} (e_j) \rangle$.  Note that the operators $\gamma(e^i)\gamma(\nu^j)$ all anti-commute with $r = \gamma(e^1)\gamma(e^2)$ and hence interchange $V_+$ and $V_-$. Thus in particular, we note that $\nabla e_k q_j$ is in $V_+$ for $j=3,4$. This will greatly simplify the computation below.

To check when the tangent space at $(x (\mathbf u_0), \, t_1 q_1 + t_2 q_2 + \psi)$ is Cayley, we need to determine when the purely imaginary $4$-fold octonion product $\imag (E_1 \times E_2 \times F_1 \times F_2)$ vanishes. This $4$-fold product is given by
\begin{equation*}
\imag (a \times b \times c \times d) = \imag \left( \bar a ( b ( \bar c d) ) \right)
\end{equation*}
whenever $a,b,c,d$ are \emph{orthogonal} octonions. Here $\bar a$ is the conjugate of $a$. For non-orthogonal arguments we can write them in terms of an orthogonal basis and expand by multilinearity. (See~\cite{HL} Section IV.1.C for details.) Without loss of generality we can assume that at the point $ x(\mathbf u_0)$, we have chosen our coordinates so that $e^1 = \oee$ and $e^2 = \oie$ with respect to the identification $T_x (\spm (\R^4)) \cong \Oc$, where $T_x (\R^4) \cong \Qu \oee$ and $(\spm)_x \cong \Qu$. Similarly we can also take $\nu^1 = \oje$ and $\nu^2 = \oke$. From this choice it follows that $\jm = \oee (\oie) = \oi$. Thus at this point $x$, the oriented orthonormal basis for $V_+$ is just $q_1 = \oo$, $q_2 = \oi$, and the orthonormal basis for $V_-$ is $q_3 = \oj$ and $q_4 = \ok$. Now we can compute (using the octonion multiplication table) and find:
\begin{align*}
\gamma(e^1)\gamma(\nu^1)q_1 & = \oj, & \gamma(e^1)\gamma(\nu^1)q_2 & = \ok, \\ \gamma(e^1)\gamma(\nu^2)q_1 & = \ok, & \gamma(e^1)\gamma(\nu^2)q_2 & = -\oj, \\
\gamma(\nu^1)\gamma(e^2)q_1 & = \ok, & \gamma(\nu^1)\gamma(e^2)q_2 & = -\oj, \\ \gamma(\nu^2)\gamma(e^2)q_1 & = -\oj, & \gamma(\nu^2)\gamma(e^2)q_2 & = -\ok.
\end{align*}
Substituting the above expressions into~\eqref{cayleytangentvectoreq2} and using~\eqref{cayleytangentvectoreq1}, we find that the tangent vectors at the point $(x (\mathbf u_0), \, t_1 q_1 + t_2 q_2 + \psi)$ are given by
\begin{align*}
E_1 & = \, \oee + \frac{t_1}{2} \oi \left( (A^1_{12} - A^2_{11}) \oj + (A^1_{11} + A^2_{12} ) \ok \right) + \frac{t_2}{2} \oi \left( (-A^1_{11} - A^2_{12} ) \oj + (A^1_{12} - A^2_{11}) \ok \right) \\ & \qquad {} + \alpha_1 \oj + \alpha \nabla_{e_1} q_3 + \beta_1 \ok + \beta \nabla_{e_1} q_4, \\ E_2 & = \, \oie + \frac{t_1}{2} \oi \left( (A^1_{22} - A^2_{12}) \oj + (A^1_{12} + A^2_{22} ) \ok \right) + \frac{t_2}{2} \oi \left( (-A^1_{12} - A^2_{22}) \oj + (A^1_{22} - A^2_{12} ) \ok \right) \\ & \qquad {} + \alpha_2 \oj + \alpha \nabla_{e_2} q_3 + \beta_2 \ok + \beta \nabla_{e_2} q_4, \\ F_1 & = \, \oo, \\ F_2 & = \, \oi.
\end{align*}
where $\alpha_i = \frac{\partial \alpha}{\partial u^i}$ and $\beta_i = \frac{\partial \beta}{\partial u^i}$.
As mentioned above, $\nabla_{e_k} q_3$ and $\nabla_{e_k} q_4$ are both in $V_+$, which is spanned by $\oo$ and $\oi$. Since the $4$-fold product $\imag (E_1 \times E_2 \times F_1 \times F_2)$ is alternating, and since $F_1 = \oo$ and $F_2 = \oi$, we can drop the terms $\nabla_{e_k} q_3$ and $\nabla_{e_k} q_4$ from $E_1$ and $E_2$ for the purposes of computing $\imag (E_1 \times E_2 \times F_1 \times F_2)$. After a tedious computation using the table in Appendix~\ref{octtablesec}, the result is
\begin{equation} \label{cayleytempeq}
\imag (E_1 \times E_2 \times F_1 \times F_2) = (C_4 C_1 - C_3 C_2) \oi +(C_1 - C_3) \oje + (C_2 - C_4) \oke,
\end{equation}
where
\begin{align*}
C_1 & = \, \frac{t_1}{2} ( A^1_{22} - A^2_{12} ) - \frac{t_2}{2} (A^1_{12} + A^2_{22}) + \beta_2, \\
C_2 & = \, \frac{t_1}{2} ( A^1_{12} + A^2_{22} ) + \frac{t_2}{2} (A^1_{22} - A^2_{12}) - \alpha_2, \\
C_3 & = \, -\frac{t_1}{2} ( A^1_{11} + A^2_{12} ) - \frac{t_2}{2} (A^1_{12} - A^2_{11}) + \alpha_1, \\
C_4 & = \, \frac{t_1}{2} ( A^1_{12} - A^2_{11} ) - \frac{t_2}{2} (A^1_{11} + A^2_{12}) + \beta_1. \\
\end{align*}
For~\eqref{cayleytempeq} to vanish, we must have $C_1 = C_3$ and $C_2 = C_4$, so the coefficient of $\oi$ will vanish automatically. The last two terms can be simplified to
\begin{equation*}
\left( \frac{t_1}{2} (A^1_{11} + A^1_{22}) - \frac{t_2}{2} (A^2_{11} + A^2_{22}) + (\beta_2 - \alpha_1) \right) \oje + \left( \frac{t_1}{2} (A^2_{11} + A^2_{22}) + \frac{t_2}{2} (A^1_{11} + A^1_{22}) - (\alpha_2 + \beta_1) \right) \oke.
\end{equation*}
This clearly vanishes for all $t_1, t_2$ if and only if $\tr A^{\nu_1} = \tr A^{\nu_2} = 0$ and $\alpha_1 = \beta_2$ and $\alpha_2 = - \beta_1$. The first two conditions say $L$ is minimal in $\R^4$. By an argument entirely analogous to that at the end of the proof of Theorem~\ref{assthm}, the last two conditions are equivalent to $\psi$ being a holomorphic section of $V_-$.
\end{proof}

\section{An explicit example} \label{examplesec}

In this section we will content ourselves with a family of explicit examples of a ``twisted'' associative subbundle of $\R^7$. Recall that a complex one-dimensional submanifold of $\R^4 \cong \C^2$ is a minimal surface. Consider the holomorphic surface $(x, y, u(x,y), v(x,y))$ in $\R^4$ where the Cauchy-Riemann equations $u_x = v_y$ and $u_y = - v_x$ are satisfied. Then one can construct the vector
$\omega^1 = e^1 \wedge e^2 - \nu^1 \wedge \nu^2$ in $\Lambda^2_-$ and it turns out to be (using the Cauchy-Riemann equations to simplify):
\begin{equation} \label{exampleo1eq}
\omega^1 \, = \, \frac{1}{1 + {|\nabla u|}^2} \left( 1 - {|\nabla u|}^2 , 2 u_y, 2 u_x \right).
\end{equation}
Similarly, one can compute that
\begin{equation} \label{exampleo23eq}
\begin{aligned}
\omega^2 & = \, \frac{1}{1 + {|\nabla u|}^2} \left(- 2 u_y , 1 + u_x^2 - u_y^2 , -2 u_x u_y \right), \\
\omega^3 & = \, \frac{1}{1 + {|\nabla u|}^2} \left(- 2 u_x ,  -2 u_x u_y, 1 - u_x^2 + u_y^2 \right).
\end{aligned}
\end{equation}
Hence Theorem~\ref{assthm} gives the following associative submanifold of $\mathbb R^7$:
\begin{equation} \label{exampleeq1}
(t \omega^1 + \alpha(x,y) \omega^2 + \beta(x,y) \omega^3, x, y, u(x,y), v(x,y) ),
\end{equation}
where $\alpha \omega^2 + \beta \omega^3$ is a holomorphic section of the holomorphic line bundle $F$ over $L$. Since we have not chosen an adapted basis satisfying~\eqref{assumptioneq}, the equations for holomorphicity are \emph{not} $\alpha_x = \beta_y$ and $\alpha_y = -\beta_x$. Instead, we need to again follow the argument in equation~\eqref{holomorphiceq}, but this time we cannot use Proposition~\ref{selfdualformulasprop}. However, since the $\omega^j$'s have unit length, the covariant derivatives $\nabla_{e_k} \omega^j$ have no component in the $\omega^j$ direction. Thus we find
\begin{equation*}
\begin{aligned}
(e_1 + i e_2) \left( \nabla^F \sigma \right)  & = \, \nabla^F_{e_1} \sigma + J( \nabla^F_{e_2} \sigma) \\ & = \, \pi_F (\nabla_{e_1} \sigma) + J( \pi_F (\nabla_{e_2} \sigma) ) \\ & = \, \pi_F (\alpha_1 \omega^2 + \alpha \nabla_{e_1} \omega^2 + \beta_1 \omega^3 + \beta \nabla_{e_1} \omega^3) + J ( \pi_F (\alpha_2 \omega^2 + \alpha \nabla_{e_2} \omega^2 + \beta_2 \omega^3 + \beta \nabla_{e_2} \omega^3)) \\ & = \, \alpha_1 \omega^2 + \beta_1 \omega^3 + \alpha \langle \nabla_{e_1} \omega^2 , \omega^3 \rangle \omega^3 + \beta \langle \nabla_{e_1} \omega^3 , \omega^2 \rangle \omega^2 \\ & \qquad {}+ J(\alpha_2 \omega_2 + \beta_2 \omega^3 + \alpha \langle \nabla_{e_2} \omega^2 , \omega^3 \rangle \omega^3 + \beta \langle \nabla_{e_2} \omega^3 , \omega^2 \rangle \omega^2) \\ & = \, \left( \alpha_1 - \beta_2 + \beta \langle \nabla_{e_1} \omega^3, \omega^2 \rangle - \alpha \langle \nabla_{e_2} \omega^2 , \omega^3 \rangle \right) \omega^2 \\ & \qquad {}+ \left( \beta_1 + \alpha_2 + \alpha \langle \nabla_{e_1} \omega^2, \omega^3 \rangle + \beta \langle \nabla_{e_2} \omega^3, \omega^2 \rangle \right) \omega^3,
\end{aligned}
\end{equation*}
and thus $\sigma = \alpha \omega^2 + \beta \omega^3$ is holomorphic if and only if
\begin{equation} \label{holomorphiceq2}
\begin{aligned}
\alpha_x - \beta_y & \, = \, -\beta \langle \nabla_{e_1} \omega^3, \omega^2 \rangle + \alpha \langle \nabla_{e_2} \omega^2 , \omega^3 \rangle, \\ \alpha_y + \beta_x & \, = \, -\alpha \langle \nabla_{e_1} \omega^2, \omega^3 \rangle - \beta \langle \nabla_{e_2} \omega^3, \omega^2 \rangle.
\end{aligned}
\end{equation}
Hence, if $\alpha$ and $\beta$ satisfy equations~\eqref{holomorphiceq2} then~\eqref{exampleeq1} gives an associative submanifold of $\R^7$. For a concrete choice, let us take $u + i v = e^z$, and thus $u(x,y) = e^x \cos y$ and $v(x,y) = e^x \sin y$. Then one can check that equations~\eqref{holomorphiceq2} become
\begin{align*}
\alpha_x & = \, \beta_y - \frac{2 e^{2x}}{1 + e^{2x}} \alpha, \\ \alpha_y & = \, - \beta_x + \frac{2 e^{2x}}{1 + e^{2x}} \beta.
\end{align*}
We can find one simple family of solutions by assuming that $\alpha$ and $\beta$ are independent of $y$. These can then be integrated to obtain
\begin{equation*}
\alpha \, = \, \frac{C}{1 + e^{2x}}, \qquad \qquad \beta \, = \, K(1 + e^{2x}),
\end{equation*}
for some constants $C$ and $K$. Substituting these into~\eqref{exampleeq1}, using~\eqref{exampleo1eq} and~\eqref{exampleo23eq}, and simplifying, we obtain
\begin{align*}
x^1 & \, = \, \frac{t - t e^{4x} + 2 C e^x \sin y - 2K e^x \cos y \, (1 + e^{2x})^2}{1 + 2 e^{2x} + e^{4x}}, \\ x^2 & \, = \,  \frac{- 2 t e^x \sin y - 2 t e^{3x} \sin y + C(1 + 2 e^{2x} \cos 2 y) + K(1 + e^{2x})^2 e^{2x} \sin 2y}{1 + 2 e^{2x} + e^{4x}}, \\ x^3 & \, = \, \frac{2 t e^x \cos y + 2 t e^{3x} \cos y + C e^{2x} \sin 2 y + K(1 + e^{2x})^2(1 - e^{2x} \cos 2 y)}{1 + 2 e^{2x} + e^{4x}}, \\ x^4 & \, = \, x, \\ x^5 & \, = \, y, \\ x^6 & \, = \, e^x \cos y, \\ x^7 & \, = \, e^x \sin y,
\end{align*}
as an explicit example of a \emph{non-ruled} associative submanifold of $\R^7$. When $C = K = 0$, this reduces to the (ruled) example of Section 5.2 of~\cite{IKM}. Similarly lengthy computations can also be done in the coassociative and Cayley cases.

\section{Conclusion} \label{conclusionsec}

The results in the present paper demonstrate that noncompact calibrated ``subbundles'' of Euclidean space, which in particular are \emph{ruled} calibrated submanifolds, have a rich deformation theory, that includes deformations through \emph{non-ruled} calibrated submanifolds. It is an interesting question to study whether or not there exist any other deformations which are not of this type. A general theorem on the deformation theory of ruled calibrated submanifolds is still lacking, although some work has been done by Joyce~\cite{J, JS} and Lotay~\cite{Lo1, Lo2}, among others.

In addition, given the likelihood that the ``twisted special Lagrangian subbundle'' equations~\eqref{SLthmeq} admit solutions in which $L^p$ is \emph{not} austere in $\R^n$, especially for $p \geq 3$, this would give a negative answer to the ``Question 5.3'' posed at the end of~\cite{KM}. It is interesting to note, however, that in the case of the exceptional calibrations, Theorems~\ref{assthm},~\ref{coassthm}, and~\ref{cayleythm} in the present paper show that the ``base'' of the twisted calibrated subbundles \emph{remains} of the same type as in the untwisted case: minimal or negative superminimal. Hence these constructions do not contradict ``Question 5.3'' of~\cite{KM} in the case of exceptional calibrations. While there is admittedly little evidence for ``Question 5.3'', it certainly remains of great interest to study possible local models for the intersections of calibrated submanifolds inside compact manifolds of special holonomy.

It would also be interesting to determine to what extent these ``twisted'' constructions extend to the cohomogeneity one special holonomy metrics considered in~\cite{KM}. It seems likely that they do. In fact, it is conceivable that these results may hold in general manifolds of special holonomy, although if so, proofs would probably need somewhat different techniques.

\appendix

\section{An identity involving the symmetric polynomials of a matrix} \label{appendixsec}

In this appendix we derive an identity that is used in Theorem~\ref{SLthm}.

\begin{lemma} \label{appendixlemma}
Let $A$ and $B$ be $p\times p$ real matrices. Then the following identity holds:
\begin{equation*}
\sum_{k=0}^p t^k \left( {\left. \frac{d^j}{ds^j} \right|}_{s=0} \sigma_k(B + sA) \right) \, = \, j! \, t^j \, \det(I + t B) \, \sigma_j (A (I + tB)^{-1} )
\end{equation*}
for all $t \in \C$ such that $I + t B$ is invertible, and all $0 \leq j \leq p$.
\end{lemma}
\begin{proof}
We begin by applying~\eqref{sigmakeq} to $B + sA$, and expanding:
\begin{align*}
\sum_{k=0}^p t^k \sigma_k(B + sA) & = \, \det(I + t(B + s A)) \\ & = \, \det( ( I + st A (I + tB)^{-1}) (I + t B) ) \\ & = \, \det(I + t B) \det(I + (st) A (I + tB)^{-1}) \\ & = \, \det(I + t B) \sum_{k=0}^p s^k t^k \sigma_k ( A (I + tB)^{-1} ).
\end{align*}
Now we differentiate the above equation $j$ times with respect to $s$, and note that on the right hand side, only the terms with $k \geq j$ survive:
\begin{equation*}
\sum_{k=0}^p t^k \left( \frac{d^j}{ds^j} \sigma_k(B + sA) \right) \, = \, \det(I + t B) \sum_{k=j}^p \frac{k!}{(k-j)!} s^{k-j} t^k \sigma_k (A (I + tB)^{-1}).
\end{equation*}
Setting $s=0$ above, only the term with $k=j$ survives on the right hand side. We obtain
\begin{equation*}
\sum_{k=0}^p t^k \left( {\left. \frac{d^j}{ds^j} \right|}_{s=0} \sigma_k(B + sA) \right) \, =  j! \, t^j \, \det(I + t B) \, \sigma_j (A (I + tB)^{-1} ).
\end{equation*}
\end{proof}

\section{Octonion multiplication table} \label{octtablesec}

Here is a multiplication table for the octonions $\mathbb O$. The table corresponds to multiplying the element in the corresponding row on the left of the element in the corresponding column. For example $\oi \cdot \oj = \ok$.

\begin{table}[h]
\begin{center}
\begin{tabular}{|c||c|c|c|c|c|c|c|c|} 
\hline {} & \oo & \oi & \oj & \ok & \oee & \oie & \oje & \oke \\ \hline \hline \oo & \oo & \oi & \oj & \ok & \oee & \oie & \oje & \oke \\ \hline \oi & \oi & -\oo & \ok & -\oj & \oie & -\oee & -\oke & \oje \\ \hline \oj & \oj & -\ok & -\oo & \oi & \oje & \oke & -\oee & -\oie \\ \hline \ok & \ok & \oj & -\oi & -\oo & \oke & -\oje & \oie & -\oee \\ \hline \oee & \oee & -\oie & -\oje & -\oke & -\oo & \oi & \oj & \ok \\ \hline \oie & \oie & \oee & -\oke & \oje & -\oi & -\oo & -\ok & \oj \\ \hline \oje & \oje & \oke & \oee & -\oie & -\oj & \ok & -\oo & - \oi \\ \hline \oke & \oke & -\oje & \oie & \oee & -\ok & -\oj & \oi & -\oo \\ \hline
\end{tabular}
\end{center}
\end{table}

\end{document}